\documentclass[a4paper,12pt]{article}
 
\usepackage{xcolor}

\usepackage[utf8]{inputenc}

\usepackage{amsmath, amsfonts, amssymb}
\usepackage{multicol}
\usepackage{enumerate}
\usepackage{a4wide}
\usepackage{graphicx}
\usepackage{multicol}
\usepackage[T1]{fontenc}
\usepackage[sc,osf]{mathpazo}
\usepackage[scaled]{beramono}
\usepackage{tgpagella}
\usepackage{hyperref}

\newtheorem{theorem}{Theorem}[section]
\newtheorem{lemma}[theorem]{Lemma}
\newtheorem{corollary}[theorem]{Corollary}
\newtheorem{proposition}[theorem]{Proposition}

\newtheorem{definition}[theorem]{Definition}

\newenvironment{proof}{\trivlist\item[\hskip \labelsep{\emph{\textbf{Proof}}}:]}{\nopagebreak \hfill $\Box$ \endtrivlist}

\def\R{\mathbb{R}}
\def\r3{\mathbb{R}^3}
\def\H{\mathcal{H}}
\def\s{\mathbb{S}}
\def\h{\mathfrak{h}}
\def\w{\textbf{w}}

\title{Compact surfaces with boundary with  prescribed mean curvature  depending on the Gauss map} 
\author{Antonio Bueno\\
\small Departamento de Ciencias\\
 \small Centro Universitario de la Defensa de San Javier\\
 \small  E-30729 Santiago de la Ribera, Spain\and
 Rafael L\'opez\\
\small  Departamento de Geometr\'{\i}a y Topolog\'{\i}a\\
\small  Universidad de Granada, Granada, Spain}
 \date{}

\begin{document}

\maketitle
  \begin{abstract}
Given a $C^1$ function $\H$ defined in the unit sphere $\s^2$, an $\H$-surface $M$ is a surface in the Euclidean space $\R^3$ whose mean curvature $H_M$ satisfies $H_M(p)=\H(N_p)$, $p\in M$, where $N$ is the Gauss map of $M$.   Given a closed simple curve $\Gamma\subset\R^3$ and a function $\H$, in this paper we investigate the geometry of compact $\H$-surfaces  spanning $\Gamma$ in terms of $\Gamma$.  Under mild assumptions on $\H$, we prove non-existence of closed $\H$-surfaces, in contrast with the classical case of constant mean curvature. We give conditions on $\H$ that ensure that if $\Gamma$ is a circle, then $M$ is a rotational surface. We also establish the existence of estimates of the area of $\H$-surfaces in terms of the height of the surface.		
  \end{abstract}
  

%%%%%%%%%%%%%%%%%%%%%%%%%%%%%
\section{Introduction}
%%%%%%%%%%%%%%%%%%%%%%%%%%%%%%

In 1910, Bernstein considered the prescribed mean curvature equation
\begin{equation}\label{q00}
\mathcal{M}(u):=\mbox{div}\left(\frac{Du}{\sqrt{1+|Du|^2}}\right)=(1+|Du|^2)^m,
\end{equation}
where $m$ is an integer, and $u=u(x)$ is a function defined in some domain of $\R^2$ \cite{be1,be2}. The left-hand side $\mathcal{M}(u)$ is the known mean curvature operator. In \cite[p. 240]{be1}, and when $m=-1/2$, he called  Eq. \eqref{q00}   {\it ``l'\'equation des surfaces, dont la courbure en chaque point is proportionnelle (\'egale) au cosinus de l'angle de la normale en ce point avec l'axe des $z$''}. The unit normal vector field on a surface in Euclidean space $\R^3$ is known as the Gauss map. Motivated by Bernstein, in this paper  we study compact surfaces in Euclidean space $\R^3$ whose mean curvature is a function of the Gauss map. In PDE terminology, this equation can be viewed as the  prescribed mean curvature equation  
\begin{equation}\label{q0}
\mathcal{M}(u) =\mathcal{F}(Du),
\end{equation}
in subsets of $\R^2$ , where the function $\mathcal{F}\colon\R^2\to\R$ is given.  Although the literature is enormous replacing $\mathcal{F}$ by a function of type $\mathcal{F}=\mathcal{F}(x,u)$, the case considered in Eq. \eqref{q0} is rather smaller. It was Bernstein itself who proved solvability of the Dirichlet problem  of \eqref{q00} for arbitrary boundary data in convex analytic domains, provided $m\leq -1/2$. 

Sixty years later, Serrin revisited this  equation in his seminal paper \cite{se0}. Although this article is rather known by the solvability of the Dirichlet problem for the constant mean curvature equation, Serrin considers many other types of quasilinear elliptic equations. In \cite[pages 477-8]{se0}, he investigated again the Dirichlet problem of the equation studied by Bernstein, changing analyticity of the domain by smoothness as well as he generalized the results to arbitrary dimensions.  More recently, equations of type \eqref{q0} have been considered, specially for the study of radial solutions and the Dirichlet problem: see, for example, \cite{ber,bf,cc,cc0,cc1,coo,el,mx,mz}.

Returning to Eq. \eqref{q00}, it is surprising, even intriguing, that the case $m=-1/2$, that is
\begin{equation}\label{e-t}
\mbox{div}\left(\frac{Du}{\sqrt{1+|Du|^2}}\right)=\frac{1}{\sqrt{1+|Du|^2}},
\end{equation}
has acquired considerable interest in recent times.  Solutions of Eq. \eqref{e-t} are called translating solitons of the mean curvature flow which evolve purely by translations. The theory of the mean curvature flow is of high activity  in the last two decades due to the pionnering work of Huisken  \cite{hu}. Translating solitons are eternal solutions of the mean curvature flow in the sense that   their evolution exists for all times $-\infty<t<\infty$. These surfaces are of special interest since they arise as a type of singularity in the mean curvature flow. Notice that the flow can develop singularities, that is, solutions which become non-smooth in finite time. The so-called type II singularities are related with translating solitons because if the initial surface in the flow is mean convex, and it develops a type II singularity, then the limit surface, after rescaling the flow, is a translating soliton \cite{hs}.

A strategy to address with Eq. \eqref{q0} is that the Gauss map of the graph $M$ of $u$, when regarded as a surface in Euclidean space $\R^3$, is 
$$N_p=\frac{(-Du,1)}{\sqrt{1+|Du|^2}}(x),\quad p=(x,u(x))\in M.$$
If $e_3=(0,0,1)$, then $\langle N,e_3\rangle=1/\sqrt{1+|Du|^2}$, where $\langle,\rangle$ is the Euclidean metric of $\R^3$. This says us that the right-hand side in \eqref{q0} can be expressed as 
$\mathcal{F}(Du)=\H(N_p)$, where $\H$ is a certain function defined on the unit sphere $\s^2$. This motivates the following definition:
\begin{definition} \label{def1}
Let $\H$ be a $C^1$ function in $\s^2$. A surface $M$ in $\R^3$ is said to be an $\H$-surface if its mean curvature $H_M$ satisfies  
\begin{equation}\label{h1}
H_M(p)=\H(N_p),\hspace{.5cm} \forall p\in M.
\end{equation}
\end{definition}

Notice that Eq. \eqref{h1} is a prescribed mean curvature equation depending on the Gauss map of the surface. This dependence on the Gauss map makes that surfaces defined by \eqref{h1} can be viewed as a type of anisotropic mean curvature equation \cite{ber}. Studying hypersurfaces in Lorentz-Minkowski space, the authors in \cite{coo} name ``gradient dependent prescribed mean curvature equation'' if $\mathcal{F}$ is of type $\mathcal{F}(x,u,Du)$. A last remark about the above definition is that if we replace in \eqref{h1} the mean curvature by the Gauss curvature, the solutions of the corresponding equations are just the solutions of the well-known Minkowski problem for ovaloids \cite{Min}.

 For particular choices of the prescribed function, some deeply studied geometric theories arise. We highlight the following examples:
\begin{enumerate}
\item Surfaces of constant mean curvature (CMC surfaces for short). Here $\H=H_0\in\R$ is a constant. If $H_0=0$ we have minimal surfaces.
\item $\lambda$-translators. The prescribed function is $\H(x)=\langle x,\w\rangle+\lambda$, where $\w\in\s^2$ and $\lambda\in\R$. The vector $\w$ is called the density vector. When $\lambda=0$, we have translating solitons   of the mean curvature flow.
\item Solitons of powers of the mean curvature flow. The mean curvature of these surfaces is $H_M(p)=\langle N_p,v\rangle^\alpha+\lambda$, where $\alpha>0$ and $\lambda\in\R$. When $\alpha>1$, these surfaces are $\H$-surfaces for $\H(x)=\langle x,v\rangle^\alpha+\lambda$. Notice that $\H\in C^1(\s^2)$.
 \end{enumerate}

 The theory of complete, non-compact $\H$-surfaces has been recently being developed by the first author in joint work with Gálvez and Mira \cite{BGM1,BGM2}, taking as starting point the theories of CMC surfaces and  translating solitons. See also \cite{BuOr1,lo} for a study of invariant $\lambda$-translators and \cite{BuLo,BuOr2} for a more general linear Weingarten prescribed curvature problem in $\r3$.

Although the main properties of complete $\H$-surfaces have been exhibited in the aforementioned papers, less is known about compact $\H$-surfaces with boundary. When $\H$ is constant, there is a large literature about this topic; see e.g. \cite{lo0} for an outline of the development of this theory. This line of inquiry has been also studied for translating solitons \cite{Pyo} and $\lambda$-translators \cite{Lop}, but as far as we know, there is not a dedicated research for an arbitrary prescribed function $\H$. To fix the terminology, let $\Gamma\subset\r3$ a closed curve and $\psi:M\rightarrow\r3$ be an immersion of a surface $M$ with boundary $\partial M$. We say that $\Gamma$ is the boundary of $M$ (if the immersion $\psi$ is understood) if $\psi_{|\partial M}$ is a diffeomorphism onto $\Gamma$. When we regard $M$ immersed in $\r3$, we will not distinguish between $\Gamma$ and $\partial M$, provided that both are diffeomorphic, and we will commonly say that $M$ spans $\Gamma$. The main problem addressed in this paper is to investigate the geometry of compact $\H$-surfaces in terms of their boundary. For example, it is desirable that $M$ inherits the symmetries of $\Gamma$, or that for a fixed curve $\Gamma$, not every $\H$ is admissible for the existence of an $\H$-surface with boundary $\Gamma$.

Based on the  above examples, in this paper we study the solutions of \eqref{h1}  using ideas and techniques coming from the theory of surfaces with constant mean curvature. This adds a geometric viewpoint to the problems related with \eqref{q0}. For example, we will use   the moving plane method introduced by Alexandrov, who proved  that spheres are the only compact  constant mean curvature surfaces embedded in $\R^3$, \cite{Ale}. Later, Serrin used the same method  in the study of elliptic overdetermined problems \cite{se}. 

In virtue of Eq. \eqref{h1}, the following are three fundamental properties of $\H$-surfaces that will allow us to approach these questions: (1) $\H$-surfaces are invariant under Euclidean translations; (2) $\H$-surfaces are locally solutions of a quasilinear, elliptic PDE and in particular the maximum principle applies; and (3) any symmetry of $\H$ in $\s^2$ induces a linear isometry of $\r3$ that sends $\H$-surfaces into $\H$-surfaces. Nonetheless, the arbitrariness of $\H$ entails further incoming difficulties that have to been taken into account. For example, in general, $\H$-surfaces are not   solutions to a variational problem involving geometric measures such as the area or volume. On the other hand,   we need to take into account the loss of symmetry  of Eq. \eqref{h1}. For instance, the reflection of an $\H$-surface with respect to an arbitrary plane is not necessarily an $\H$-surface, unless this plane is a symmetry plane of $\H$. In some of the results that we will obtain, we will compare the analogue situation for CMC surfaces, emphasizing the differences, if any. 

The organization of this paper is as follows. In Section \ref{sec2}, we state the maximum and comparison principles, which are the cornerstone to obtain the majority of the forthcoming results. We analyze the existence and non-existence of closed $\H$-surfaces, proving in Prop. \ref{lemacerradas} that the condition $\H\neq0$ is necessary for the existence of a closed $\H$-surface. Nonetheless, in contrast to the CMC case, the condition $\H\not=0$ is not sufficient for the existence of a closed $\H$-surface, as revealed in Prop. \ref{propnoexistencerradas}. In Section \ref{sec3}, we investigate whether the symmetries of $\Gamma$ are inherited to a compact $\H$-surface spanning $\Gamma$.  The usual method to study this problem is the so-called Alexandrov reflection technique, and one of the major issues that we can find is that the compact surface with planar boundary may have points lying at both sides of the plane where its boundary lies, just as in the CMC case. In Section \ref{sec4},  we give sufficient conditions to guarantee that a compact $\H$-surface lies at one side of the plane containing its boundary:  Ths. \ref{t-k}  and \ref {thHsuperficieaunlado}. As a particular but important case of $\H$-surfaces, these results will be applied to $\lambda$-translators when $|\lambda|\leq 1$. Finally, in Section \ref{sec5}, we establish an estimate of the area of a compact $\H$-surface with planar boundary in terms of the height to the boundary (Th. \ref{thalturaarea}).

%%%%%%%%%%%%%%%%%%%%%%%%%%%%%%%%%%%%%%%%%%%%%%%%%%
\section{The maximum principle for $\H$-surfaces and first consequences}\label{sec2}
%%%%%%%%%%%%%%%%%%%%%
 
 In this section, we state the maximum principle for $\H$-surfaces, extending the well-known situation of CMC surfaces. Let $(x_1,x_2,x_3)$ be the canonical coordinates of $\R^3$ and let $\{e_1,e_2,e_3\}$ denote the usual basis of $\R^3$. Given an $\H$-surface in $\r3$, we can locally express it as the graph of a function $u$ over each tangent plane, being $u$ a solution to an elliptic, second order quasilinear PDE. Specifically, if $u=u(x_1,x_2)$, then $u$ satisfies the equation
\begin{equation}\label{eqH}
\mathrm{div}\left(\frac{Du}{\sqrt{1+|Du|^2}}\right) = 2 \H\left(\frac{(-Du,1)}{\sqrt{1+|Du|^2}}\right),
\end{equation}
where $\mathrm{div},D$ denote respectively the divergence and gradient operators on $\R^2$. In this spirit, Eq. \eqref{eqH} can be expressed as $Q(u)=0$, where
$$Q(u)=\sum_{i,j=1}^2a_{ij}(x,u,Du)\frac{\partial^2 u}{\partial x_i\partial x_j}+b(x,u,Du),\quad x=(x_1,x_2).$$
As a consequence, the difference of two solutions of $Q(u)=0$ satisfies a linear elliptic equation. The term $b(x,u,Du)$ gathers the right-hand side of \eqref{eqH}. In order to apply the Hopf maximum principle in its interior and boundary versions, it is required that $b$ is continuously differentiable with respect to the variable $Du$ \cite[Th. 92]{gt}, hence we require the function $\H$ to be $C^1$ in Definition \ref{def1}. A geometric version of the maximum principle is the following:

\begin{lemma}[Maximum principle of $\H$-surfaces]\label{ppiomax}
Let $M_1,M_2$ be two $\H$-surfaces, possibly with smooth boundary. Assume that one of the following two conditions holds:
\begin{enumerate}
\item
There exists $p\in \mathrm{int}(M_1)\cap \mathrm{int}(M_2)$ such that $(N_1)_p=(N_2)_p$, where $N_i$ is the Gauss map of $M_i$.
\item
There exists $p\in \partial M_1\cap \partial M_2$ such that $(N_1)_p=(N_2)_p$ and $\nu_1(p)=\nu_2(p)$, where $\nu_i$ denotes the interior unit conormal of $\partial M_i$.
\end{enumerate}
If $M_1$ lies  at one side of $M_2$ around $p$, then $M_1=M_2$.
\end{lemma}

In the theory of CMC surfaces, the maximum principle has a stronger version when applied to minimal surfaces, the so-called \emph{tangency principle}. It states that two minimal surfaces cannot have a common tangent point and one lying at one side of the other, regardless of the orientation chosen. This is because the orientation of a minimal surface can be reversed without changing the underlying PDE. In this spirit, we have the corresponding tangency principle of $\H$-surfaces for a certain type of functions $\H$.

\begin{corollary}[Tangency principle of $\H$-surfaces]\label{ppiotangencia}
Let $\H$ such that $\H(-x)=-\H(x)$, $x\in\s^2$. If $M_1,M_2$ are two $\H$-surfaces and $M_1$ lies at one side of $M_2$, then $M_1=M_2$.
\end{corollary}
\begin{proof} Suppose that $M_1$ lies at one side of $M$ around some common (interior or boundary) point $p_0$. If $(N_1)_{p_0}=(N_2)_{p_0}$, then we apply the maximum principle of $\H$-surfaces to get the result. Otherwise, that is, when $(N_1)_{p_0}=-(N_2)_{p_0}$, we reverse the orientation of $M_1$. Then $M_1$ is also an $\H$-surface with this orientation because its mean curvature $H'_{M_1}$ satisfies that for all $p\in M_1$, we have 
$$H'_{M_1}(p)=-H_{M_1}(p)=-\H((N_1)_p)=\H(-(N_1)_p).$$
Then we are again under the conditions to apply the maximum principle of $\H$-surfaces around $p_0$, obtaining the result.
\end{proof}

Some examples of functions $\H$  satisfying the condition $\H(-x)=-\H(x)$ are $\H(x)=\langle x,v\rangle^n$, for a fixed $v\in\s^2$ and $n$ an odd number ($n=1$ corresponds with translating solitons) and $\H(x)=\sin(\langle x,v\rangle)$, for a fixed $v\in\s^2$.

The following result is the classical mean curvature comparison principle but adapted to $\H$-surfaces.

\begin{lemma}[Comparison principle of $\H$-surfaces]\label{ppiocomp}
Let   $M_i$ be $\H_i$-surfaces, $i=1,2$. Assume that there exists $p\in M_1\cap M_2$ such that $(N_1)_p=(N_2)_p$. If   $M_1$ lies locally above $M_2$ around $p$, then, $\H_1((N_1)_p)\geq\H_2((N_2)_p)$.
\end{lemma}

A fundamental property that will be applied throughout the paper is the following. Let $v\in\s^2$ and assume that $\H(v)=0$. Then it is immediate that any plane oriented by $v$ is an $\H$-surface regardless of whether $\H$ is not identically $0$. 
We show next that in the same class of $\H$-surfaces, one may have different types of examples that widely differ one from the other. Fix some $\varepsilon\in(0,1)$ and let us define
$$
f(t)=\left\lbrace
\begin{array}{lll}
1,&\mathrm{if}&t\leq0\\
g(t),&\mathrm{if}&t\in[0,\varepsilon]\\
0,&\mathrm{if}&t\geq\varepsilon,
\end{array}\right.
$$
where $g(0)=1,g(\varepsilon)=0$, $g$ is decreasing and such that $f$ is of class $C^\infty$. Define $\H\in C^1(\s^2)$ by  $\H(x)=f(\langle x,e_3\rangle)$. Some $\H$-surfaces are the following:
\begin{enumerate}
\item The upper half-sphere $\s^2_+$ (or any open subset of it)   with the downwards orientation.
\item Any plane oriented by $v$, being $v\in\s^2$ such that  $\langle v,e_3\rangle\geq\varepsilon$. 
\end{enumerate}
Note that the coexistence of both planes and half-spheres in the same class of $\H$-surfaces does not contradict neither the maximum nor the mean curvature comparison principles. Indeed, if they are tangent at some point their unit normals are necessarily opposite; otherwise, they are no longer $\H$-surfaces since they do not fulfill \eqref{h1}.

We now derive some direct consequences of the maximum principle. Given $\H\in C^1(\s^2)$, we are interested in the existence of closed (compact without boundary) $\H$-surfaces. When $\H$ is a positive constant there are such closed examples, namely, spheres. This makes a huge difference when $\H$ is identically zero, since it is well-known the non-existence of closed minimal surfaces. The following result proves that it is sufficient for $\H$ to vanish at just one point of $\s^2$ in order to forbid the existence of a closed $\H$-surface.
\begin{proposition}\label{lemacerradas}
Let $\H\in C^1(\s^2)$. If $\H$ vanishes at some point, then there are no closed $\H$-surfaces in $\r3$.
\end{proposition}

\begin{proof}
By contradiction, assume that $M$ is a closed $\H$-surface. Without loss of generality, we can assume that $\H(e_3)=0$. Let $\Pi$ be the plane of equation $x_3=0$, which we orient by $e_3$.  After a vertical translation, and by compactness of $M$, we can assume that $M$ is contained in the half-space $\{x\in\R^3:x_3\geq0\}$ and that $M\cap\Pi\not=\emptyset$. Let  $p_0\in M\cap \Pi$. Since $p_0$ is an interior point of $M$, then  $N_{p_0}=\pm e_3$.

If $N_{p_0}=e_3$, we apply the maximum principle for $\H$-surfaces between the plane $\Pi$ and $M$ and we conclude that $M\subset\Pi$. This is a contradiction because  $M$ is closed. Therefore, $N_{p_0}=-e_3$. We now consider $\Pi$ oriented with the vector $-e_3$. The comparison principle of $\H$-surfaces yields $\H(-e_3)\leq0$. Moreover, $\H(-e_3)<0$ for if $\H(-e_3)=0$ we contradict the maximum principle.

By compactness again, let be $p_1\in M$ a point of maximum height to $\Pi$ and $P$ the affine tangent plane to $M$ at $p_1$. At this point we have $N_{p_1}=\pm e_3$. If $N_{p_1}=e_3$ we contradict the maximum principle. If $N_{p_1}=-e_3$, then $M$ lies locally above $P$ around $p_1$ but $\H(-e_3)<0$, which contradicts the comparison principle. In any case, we arrive to the desired contradiction.
\end{proof}

From this result, the condition $\H\not=0$ is necessary for the existence of closed $\H$-surfaces. In the same fashion as in the case that $\H$ is a non-zero constant, it would be expectable  that the condition $\H\not=0$, say $\H>0$, is also sufficient for the existence of a closed $\H$-surface. Nevertheless, the following result exhibit examples of positive functions $\H$ for which there are no closed $\H$-surfaces.

\begin{proposition}\label{propnoexistencerradas}
Let be $\H$ such that $\H(x)=h_0(x)+\lambda$, where $\lambda\in\R$ and $h_0\in C^1(\s^2)$ is not identically zero. If there is $v\in\s^2$ such that   $h_0(x)\langle x,v\rangle\geq0$ for all $x\in\s^2$, then  there are no closed $\H$-surfaces. 
\end{proposition}

\begin{proof}
By contradiction, assume that $M$ is a closed $\H$-surface with $\H=h_0+\lambda$  defined under the hypothesis of the proposition. The function $f\in C^\infty(M)$ defined by $f(p)=\langle p,v\rangle$ satisfies $\Delta f=2H\langle N,v\rangle$, where  $\Delta$ is the Laplacian in $M$. Since $\partial M=\varnothing$, the divergence theorem gives
$$
0=\int_M \H(N)\langle N,v\rangle\, dM=\int_M h_0(N)\langle N,v\rangle\, dM+\lambda\int_M\langle N,v\rangle\, dM.
$$
The second integral vanishes since $M$ is closed. Therefore, the first integral must vanish as well, but $h_0(N)\langle N,v\rangle\geq0$ and $h_0$ is not identically zero by hypothesis, arriving to the desired contradiction.
\end{proof}

Examples of positive functions $\H\in C^1(\s^2)$ under the hypothesis of Prop. \ref{propnoexistencerradas} are the following:
\begin{enumerate}
\item   Let $v\in\s^2$. For $\lambda>1$ and $n$ an odd natural number, define  $\H(x)=\langle x,v\rangle^n+\lambda$.  
\item   Let $v\in\s^2$. For $\lambda>2$, define $\H(x)=\langle x,v\rangle(1+\sin(\langle x,v\rangle))+\lambda$. 
\end{enumerate}

The last two results exhibit that under mild assumptions on $\H$, there do not exist closed $\H$-surfaces. Therefore, for these prescribed functions, any compact $\H$-surface has non-empty boundary. This makes the study of compact $\H$-surfaces with boundary not only natural, but necessary.

Given a simple closed curve $\Gamma$, our first approach to this study is to derive some restrictions for $\H$ in terms of the geometry of $\Gamma$ for the existence of a compact $\H$-surface spanning $\Gamma$. In the theory of CMC surfaces, it is known that given a closed curve $\Gamma\subset\R^3$, not all $H\in\R$ are allowed to have $H$-surfaces spanning $\Gamma$. This was proved initially by Heinz \cite{hei} and generalized later using a flux formula for CMC surfaces: see details in \cite{lo0}. For example, if $\Gamma$ is a closed simple curve, then $|H|\leq L/(2\mbox{area}(\Omega))$, where $L$ is the length of $\Gamma$ and $\Omega$ is the planar   domain bounded by $\Gamma$. As far as the authors know, there is no an analogue to the flux formula in the theory of $\H$-surfaces. Nonetheless, we are able to prove a similar necessary condition in case that the surface is a graph. 

\begin{proposition}\label{proprelacionLareaminH}
Let $\Gamma$ be a closed simple curve contained in a plane $\Pi=v^\bot$, $v\in\s^2$. If $\H>0$ is a $C^1$ function defined in $\s^2$, then a necessary condition for the existence of a compact $\H$-graph spanning $\Gamma$ is 
$$
\H_{\min}:=\min\{\H(x):x\in\s^2\} \leq\frac{L}{2\,\mathrm{area}(\Omega)},
$$
where $L$ is the length of $\Gamma$ and $\Omega\subset\Pi$ is the domain bounded by $\Gamma$. 
\end{proposition}

\begin{proof} By \eqref{eqH} and the divergence theorem, we have 
$$\int_\Omega 2(\H\circ N)\, d\Omega=\int_\Omega\mbox{div}\left(\frac{(-Du,1)}{\sqrt{1+|Du|^2}}\right)\, d\Omega= \int_{\partial\Omega}
\langle \frac{(-Du,1)}{\sqrt{1+|Du|^2}},\mathbf{n}\rangle\, ds\leq L,$$
where $\mathbf{n}$ is the unit outward normal vector of $\partial\Omega$ as planar curve of $\Pi$. Since $\H>0$, the first term in the above identities can be estimated from below by $2\H_{min}\int_\Omega\, d\Omega$, obtaining the result.
\end{proof} 

Going back to Prop. \ref{lemacerradas},  one may expect that if $\H(v)=0$,  the class of $\H$-surfaces behaves as minimal surfaces. The following result emphasizes this relationship by generalizing the  \emph{convex hull property  of minimal surfaces}, which states that any compact minimal surface must lie inside the convex hull of its boundary. 

\begin{proposition}\label{thconvexhull}
Let $\Gamma$ be a closed   curve in $\r3$. Let $\H\in C^1(\s^2)$ such that  $\H(\pm v)=0$  for some $v\in\s^2$. If $M$ is a compact $\H$-surface with boundary $\Gamma$, then  $M$ is contained in the slab
$$
\{x\in\r3:\mu_1(v)\leq \langle x,v\rangle\leq\mu_2(v)\},$$
where 
$$\mu_1= \min\{\langle p,v\rangle:\ p\in\Gamma\},\quad \mu_2(v)=\max\{\langle p,v\rangle:\ p\in\Gamma\}.
$$
\end{proposition}

\begin{proof}
After a change of coordinates, we assume $v=e_3$. Arguing by contradiction, assume that $M$ contains points whose third coordinate $x_3$ is bigger than $\mu_2(e_3)$. Consider $p\in M$ the highest point of $M$, which in particular satisfies $x_3(p)>\mu_2(e_3)$, and hence $p$ is an interior point of $M$. Moreover, $N_{p}=e_3$ or $N_p=-e_3$. We orient the (affine) tangent plane $T_pM$ by the vector $N_p$. Since $\H(N_p)=0$, then $T_pM$ is an $\H$-surface and the maximum principle implies that $M$ is included in $T_pM$: this a contradiction because $\Gamma\cap T_pM=\emptyset$.   
\end{proof}

%%%%%%%%%
\section{The Alexandrov method for $\H$-surfaces}\label{sec3}
%%%%%%%%%%%

In this section, we investigate how the geometry of a simple closed curve $\Gamma$ determines the shape of a compact $\H$-surface that spans it. A first interesting case is to investigate whether the symmetries of the boundary curve $\Gamma$ are inherited to a compact $\H$-surface.  The main tool for achieving this purpose is the so-called Alexandrov reflection technique. A major difference with the CMC setting is that in the case of an arbitrary prescribed function $\H$, we have to take into account the loss of symmetries of the resulting PDE fulfilled by the $\H$-surfaces. As a matter of fact, the reflection about an arbitrary plane no longer sends an $\H$-surface into an $\H$-surface. In this spirit, the third hypothesis of the following result is paramount.

\begin{theorem}\label{t-s}
Let $\Gamma$ be a closed simple curve contained in a plane $\Pi$. Assume that:
\begin{enumerate}
\item $\Gamma$ is symmetric with respect to the reflection about a plane $P$ orthogonal to $\Pi$,
\item $P$ separates $\Gamma$ into two graphs over the line $\Pi\cap P$.
\item $\H$ is invariant  with respect to the reflection about the vector plane parallel to $P$.
\end{enumerate}
If $M$ is a compact, embedded $\H$-surface spanning $\Gamma$ and $M$ lies at one side of $\Pi$, then $M$ is symmetric with respect to $P$.
\end{theorem}

\begin{proof}
After a change of coordinates, we assume that $\Pi$ is the plane of equation $x_3=0$ and that $P$ is the plane of equation $x_1=0$.  Let $\Omega\subset\Pi$ the domain bounded by $\Gamma$. Since $M$ is embedded, $M\cup\Omega$ defines a closed surface, possible singular along $M\cap\Omega$, that bounds a domain $W$ in $\r3$. Assume without losing generality that $M\subset\{x_3\geq0\}$. Let us orient $M$ so $N$   points towards the interior of $W$.

For any $t\in\R$, we denote $P_t$ the plane of equation $ x_1=t$. We define $M(t)^-=M\cap\{x_1\leq t\},\ M(t)^+=M\cap\{x_1\geq t\}$ and $M(t)^*$ the reflection of $M(t)^+$ about $P_t$. At this point, the third item in the hypotheses of Th. \ref{t-s} allows us to assert that $M(t)^*$   is again a  $\H$-surface where the orientation on $M(t)^*$ is  $N^*$, being $N^*=R_{P_t}(N)$. Here $R_{P_t}$ denotes the reflection about the plane $P_t$. Indeed, if $q\in M(t)^*$, then $q=R_{P_t}(p)$ for some $p\in M(t)^+$. Then
$$H_{M(t)^*}(q)=H_{M(t)^+}(p)=\H(N_p)=\H(R_{P_t}(N_p))=\H(N^*_{q}),$$
where we have used the identity $\H=\H\circ R_{P_t}$ thanks to the third item of the hypothesis. 

Although the proof of theorem is standard, by completeness we present an outline of it in order to emphasize where the third hypothesis in the statement of Th. \ref{t-s} is crucial. Arguing by contradiction, assume that $M$ is not symmetric with respect to $P$. Then, there are two points $q_1,q_2\in M\textbackslash\Gamma$ such that the line $\overline{q_1q_2}$ is orthogonal to $P$, $q_1$ and $q_2$ are in different components of $\r3\textbackslash P$ and $\mbox{dist}(q_1,P)\neq \mbox{dist}(q_2,P)$; without losing generality, assume $x_1(q_2)<0<x_1(q_1)$. If we name $x^*$ to the reflection of $x\in\r3$ about $P$, then $x_1(q_1^*)<x_1(q_2)$.

By compactness of $M$, for $t$ large enough we have $P_t\cap M=\varnothing$. We let $t$ decrease until we arrive to a first instant $t_0>0$ where there is a first contact point between $P_{t_0}$ and $M$; we know that $t_0>0$ since $x_1(q_1)>0$. The embeddedness of $M$ and since $M$ lies above $\Pi$ ensures the existence of $\varepsilon>0$ such that $\mathrm{int}(M(t)^*)\subset W$ for every $t\in(t_0-\varepsilon,t_0)$. By compactness of $M$ and by decreasing $t$, we define
$$
t_1=\inf\{t\in\R:\ \mathrm{int}(M(t)^*)\subset W\ \mathrm{and}\ M(t)^*\text{is a graph over}\ \Pi_t\}.
$$
Because $\Pi$ is a plane of symmetry of $\Gamma$, $t_1\geq0$. For $t<t_1$, whatever the reason for which $t$ fails to be the infimum, it happens $M(t)^*\not\subset\overline{W}$. In our setting, $t_1>0$ because of the existence of $q_1,q_2$ such that $x_1(q_1^*)<x_1(q_2)<0$. 

Since $M(t)^+$ is a graph over $P_t$ for $t>t_1$ and $\Gamma\cap\{x_1>0\}$ is also a graph over the line $P\cap\Pi$, we have $\partial M(t_1)^*\cap\Gamma\subset P_{t_1}$. The latter hypothesis is also crucial since it could happen that the first meeting point $p$ between $M(t_1)^*$ and $M(t_1)^-$ lies in $\Gamma$ where none of the surfaces $M(t_1)^*$ and $M(t_1)^-$ are tangent. There are two possibilities for $p$:
\begin{enumerate}
\item There exists $p\in\mathrm{int}(M(t_1)^*)\cap\mathrm{int}(M(t_1)^-)$. Since $p$ is an interior point, $M(t_1)^*$ and $M(t_1)^-$ are tangent and the orientation of both surfaces point towards the interior of $W$ and in particular they agree. The maximum principle implies that $M(t_2)^*=M(t_2)^-$ and thus $P_{t_1}$ is a plane of symmetry of $M$. However, this contradicts that $\Gamma$ is not invariant by reflections about $P_{t_1}$.
\item The surface $M$ is orthogonal to $P_{t_1}$ at some $p\in\partial M(t_1)^*\cap\partial M(t_1)^-$. This time we use the boundary version of the maximum principle to conclude that $P_{t_1}$ is a plane of symmetry of $M$, which is again a contradiction.
\end{enumerate}
This implies that $M$ is symmetric with respect to $P$, proving the result.
\end{proof}

Some   examples of prescribed functions under the hypothesis of Th. \ref{t-s} are
$$
\H(x)=\langle x,v\rangle^\alpha+\lambda,\quad  x\in\s^2,
$$
where $\alpha\geq1$ (to ensure that $\H$ is $C^1$) and $\lambda\in\R$.
Indeed, if $P$ is a vector plane containing $v$ and $R_P$ is the reflection about $P$, then for every $x\in\s^2$ we have $\H(x)=\H(R_p(x))$. In particular, if $M$ is an $\H$-surface then the reflection about any plane parallel to $P$ is again an $\H$-surface. 

For the particular case that the above symmetry property holds for every vector plane containing $v$, then $\H$ is also invariant under any rotation of $\s^2$ that preserves the $v$-direction pointwise fixed. Our purpose  is to conclude that $M$ is a rotational $\H$-surface when its boundary is a circle. In this fashion, we give the following definition.
\begin{definition}
Let be $v\in\s^2$. A function $\H\in C^1(\s^2)$ is said to be rotational about the $v$-direction  if there exists $\h\in C^1([-1,1])$ such that
$$
\H(x)=\h(\langle x,v\rangle),\hspace{.5cm} \forall x\in\s^2.
$$
\end{definition}
As a consequence, if $\H$ is rotational about some direction $v$, then rotations about any line parallel to $v$ send $\H$-surfaces into $\H$-surfaces and in particular the notion of rotational $\H$-surface is well-defined. We refer the reader to \cite{BGM1} for a deep study and classification of rotational $\H$-surfaces, including the achievement of a Delaunay-type classification theorem. Furthermore,  as a particular case of a more general study regarding a linear Weingarten relation, it was exhibited in \cite{BuLo} necessary and sufficient conditions for the existence of a rotational $\H$-surface $M^+$ (resp. $M^-$) intersecting orthogonally the rotation axis with unit normal $N=v$ (resp. $N=-v$). The following result gives a characterization of $M^\pm$.

\begin{corollary}\label{cor-2}
Let   $\H$ be a rotational functional about   $v\in\s^2$. Let $\Gamma$ be a circle in a plane $\Pi=v^\bot$ such the plane $\Pi$ is oriented with   $v$. Assume that $M$ is a compact, embedded $\H$-surface spanning $\Gamma$ and lying at one side of $\Pi$.
\begin{enumerate}
\item If $M$ lies above $\Pi$, then $\H(-v)>0$ and $M=M^-$.
\item If $M$ lies below $\Pi$, then $\H(v)>0$ and $M=M^+$.
\end{enumerate}
\end{corollary}

%%%%%%%%%
\section{Compact $\H$-surfaces with planar boundary} \label{sec4}
%%%%%%%%%%%

As a consequence of Th. \ref{t-s}, it is interesting to have conditions that ensure that an embedded compact $\H$-surface with planar boundary lies contained in one side of the boundary plane.  We will prove two results in this direction. The first one extends a similar situation for CMC surfaces \cite{Koi}. 

\begin{theorem}\label{t-k}
Let   $\Pi$ be a plane and $\Gamma\subset\Pi$ be a closed simple curve that bounds a domain $\Omega$, and let $\mathrm{ext}(\Omega)=\Pi\textbackslash\overline{\Omega}$. If $M$ is a compact, embedded $\H$-surface that spans $\Gamma$ and $M\cap\mathrm{ext}(\Omega)=\varnothing$, then $\mathrm{int}(M)$ lies contained in one of the two closed halfspaces determined by $\Pi$.
\end{theorem}

\begin{proof}
Without loss of generality, we assume that $\Pi$ is the plane of equation $x_3=0$.  Let $\s^1(r)\subset\Pi$ a circle of radius $r$ and $\s^2_+(r)$ the upper hemisphere in the half-space $\{x_3\geq0\}$, with $\partial\s^2_+(r)=\s^1(r)$, and $D(r)$ the disc bounded by $\s^1(r)$. Since $M$ is compact, let   $r$ big enough such that $M\cap\{x_3\geq0\}$ is strictly contained in the domain bounded by $D(r)\cup\s^2_+(r)$. In this setting,
$$
K=M\cup\left(D(r)\textbackslash\Omega\right)\cup\s^2_+(r)
$$
is an embedded, closed surface in $\r3$, possibly not smooth along $\Gamma\cup\s^1(r)$, that bounds a domain $W$ in $\r3$. Let $N$ be the unit normal of $M$ and take on $K$ the orientation $\widetilde{N}$ such that $\widetilde{N}_{|M}=N$. Then $\widetilde{N}$ points towards either the interior of the exterior of $W$. We will assume that $\widetilde{N}$ points towards the interior of $W$: on the contrary, the argument is similar.  

By contradiction, assume that $M$ has interior points in both sides of $\Pi$. Let $p,q\in \mathrm{int}(M)$ be the points of minimum and maximum height, respectively, to $\Pi$ with $x_3(p)<0$ and $x_3(q)\geq0$ (if $x_3(p)\leq0$ and $x_3(q)>0$ the argument is similar). Since $\widetilde{N}$ points towards the interior of $W$, then $N_p=N_q=e_3$. Let $T_pM$ and $T_qM$ be the affine tangent planes to $M$ at $p$ and $q$, oriented with unit normal $e_3$. Since $M$ lies above $T_pM$ (resp. below $T_qM$) locally around $p$ (resp. around $q$), hence the mean curvature comparison principle yields $H_M(p)\geq 0$ and $H_M(q)\leq 0$. Thus
$$
0\leq H(p)=\H(N_p)=\H(e_3)=\H(N_q)=H(q)\leq0.$$
Therefore $\H(e_3)=0$. Hence $T_pM$ and $T_qM$ are $\H$-surfaces.  The maximum principle of $\H$-surfaces applied  to $M$ and $T_pM$ implies that $M$ is included in $T_pM$. This is a contradiction because $\Gamma$ is not included in the plane $T_pM$. 
\end{proof}

An immediate application of this theorem is  for graphs.  
\begin{corollary}\label{corolariografoladoplano}
If $M$ is a compact $\H$-graph spanning a planar closed simple curve, then   lies contained in one of the two closed halfspaces determined by the boundary plane.
\end{corollary}

%%%%%

In the second result, we prove that the surface lies contained in one side of the boundary plane $\Pi=v^\bot$ if $\H$ takes a special behaviour at $v\in\s^2$.  Comparing with the hypothesis of Th. \ref{t-k}, we drop the assumption of the embeddedness of the surface.

\begin{theorem}\label{thHsuperficieaunlado}
Let $\H\in C^1(\s^2)$ and $v\in\s^2$ such that $\H(v)\H(-v)<0$ and $\Pi=v^\bot$. Let $M$ be a compact $\H$-surface whose boundary $\partial M$ is contained in the plane $\Pi=v^\bot$.  
\begin{enumerate}
\item If $\H(v)>0$, then $\mathrm{int}(M)$ lies in the halfspace $\{x\in\R^3:\langle x,v\rangle<0\}$.
\item If $\H(v)<0$, then $\mathrm{int}(M)$ lies in the halfspace $\{x\in\R^3:\langle x,v\rangle>0\}$.
\end{enumerate}
\end{theorem}
Examples of functions $\H$ satisfying the condition of this theorem are $\H(x)=\langle x,v\rangle^n$, where $n$ is an odd natural number.

\begin{proof} Without loss of generality, we assume that $v=e_3$. We prove the first item, since the second one is analogous. Assume $\H(e_3)>0$ and, by hypothesis, $\H(-e_3)<0$. Arguing by contradiction, suppose the existence of  interior points of $M$ in $\{x\in\R^3:\langle x,e_3\rangle\geq 0\}$.   Let $q=\max\{\langle p,e_3\rangle:\ p\in M\}$, in particular,    $\langle q,e_3\rangle\geq 0$. Since $q$ is an interior point, its tangent plane $T_qM$ is the plane of equation $x_3=x_3(q)$. Thus  $N_q=e_3$ or $N_q=-e_3$. 

Suppose that $N_q=e_3$. We orient $T_qM$ by the vector $e_3$ and in particular, $T_qM$ is an $\H_1$-surface where $\H_1(x)=\H(x)-\H(e_3)$. Then $T_qM$ lies above $M$ around $q$ and the comparison principle of $\H$-surfaces implies $\H_1(N_q)\geq\H(N_q)$. This says $0\geq \H(e_3)$, which is not possible. Thus necessarily $N_q=-e_3$.  Now we orient $T_qM$ by $-e_3$, being $T_qM$ an $\H_2$-surface for $\H_2(x)=\H(x)-\H(-e_3)$. Since now $M$ lies above $T_qM$ around $q$, the comparison principle yields $\H(N_q)\geq\H_2(N_q)$, that is, $\H(-e_3)\geq 0$. This contradiction proves the result. 
\end{proof}

A context where Theorem \ref{thHsuperficieaunlado} is useful is in the study of $\lambda$-translators with boundary. Recall that a surface $M$ in $\r3$ is a $\lambda$-translator if its mean curvature $H_M$ satisfies  $
H_M(p)=\langle N_p,\textbf{w}\rangle+\lambda$, where $ \textbf{w}\in\s^2$ is called the \emph{density vector} and $\lambda\in\R$.  After a change of the orientation, we can assume  $\lambda\geq0$ without losing generality.  In \cite{Lop}, the second author addressed the study of compact $\lambda$-translators with boundary. We now  improve these results, by showing that the hypothesis $0\leq \lambda\leq1$ has a paramount consequence in Th. \ref{thHsuperficieaunlado}.

\begin{corollary}\label{corlambdatranslatoraunlado}
Let $0\leq \lambda\leq1$ and $M$ be a compact $\lambda$-translator with density vector $\emph{\w}$ and   whose boundary is contained in a plane $\Pi=v^\bot$, $v\in\s^2$. If $|\langle v,\emph{\w}\rangle|\geq\lambda$, then $\mathrm{int}(M)\subset\{x\in\R^3:\mathrm{sign}(\langle v,\emph{\w}\rangle) \langle x,v\rangle\leq0\}$. Moreover, there exists $p_0\in\mathrm{int}(M)\cap\Pi\not=\emptyset$ if and only if $M$ is planar and $|\langle v,\emph{\w}\rangle|=\lambda$.

\end{corollary}
\begin{proof}
If $M$ is a $\lambda$-translator, we have  $\H(x)=\langle x,\w\rangle+\lambda$. First, assume $|\langle v,\w\rangle|>\lambda$. This condition yields $\H(v)\H(-v)=\lambda^2-\langle v,\w\rangle^2<0$, hence we are in the hypothesis of Th. \ref{thHsuperficieaunlado}. Moreover, if $\mathrm{sign}(\langle v,\w\rangle)>0$ then $\H(v)=\langle v,\w\rangle+\lambda>0$ and $M$ lies in $\{x\in\R^3:\langle x,v\rangle<0\}$. In case that $\mathrm{sign}(\langle v,\w\rangle)<0$, then  $M$ lies in $\{x\in\R^3:\langle x,v\rangle>0\}$.

For the second part, assume the existence of some $p_0\in\mathrm{int}(M)\cap\Pi$. Without losing generality, we suppose $\langle v,\w\rangle>0$, which in particular yields $\langle v,\w\rangle\geq\lambda$. From the first part, we know that $M$ lies in $\{\langle x,v\rangle\leq0\}$, hence at $p_0$ we have $N_{p_0}=\pm v$. If $N_{p_0}=v$, then $\Pi$ lies locally above $M$ around $p_0$ and the mean curvature comparison principle yields
$$
0\geq H_M(p_0)=\H(v)=\langle v,\w\rangle+\lambda>0,
$$
a contradiction. Consequently, $N_{p_0}=-v$ and $M$ lies locally above $\Pi$ around $p_0$. The comparison principle with the mean curvature gives
$$H_M(p_0)=\H(N_{p_0})=\H(-v)=-\langle v,\w\rangle+\lambda\geq0,$$
and consequently $\langle v,\w\rangle=\lambda$. Then, $\Pi$ with the orientation $-v$ is an $\H$-surface and the maximum principle yields that $M$ is planar.

If $\langle v,\w\rangle<0$, which in particular yields $\langle v,\w\rangle\leq-\lambda$, the proof is similar. This time, $M$ lies above $\Pi$ around $p_0$ and $N_{p_0}=v$ necessarily. Consequently, by the comparison principle we have $\langle v,\w\rangle+\lambda\geq0$, from where we conclude $\langle v,\w\rangle=-\lambda$. This time, the maximum principle applied to $M$ and $\Pi$ with orientation $v$ allows us to assert that $M$ is contained in $\Pi$, proving the result. \end{proof}

We can combine this corollary together Cor. \ref{cor-2} to conclude the following result of symmetry for $\lambda$-translators. We illustrate the result in the case that the boundary is a circle.

\begin{corollary}
Let $M$ be a compact embedded $\lambda$-translator with density vector $v\in\s^2$ and $|\lambda|\leq 1$. If the boundary of $M$ is a circle contained in a plane $\Pi=v^\bot$, then $M$ is an open piece of a rotational $\lambda$-translator that intersects orthogonally the rotation axis.
\end{corollary}

We finish this section with a generalization of a result due to Pyo. In \cite{Pyo} it was proved that a compact translating soliton with density $v$ spanning a circle contained in a plane $\Pi^\bot$ is a rotational surface. The techniques to proved this result can now be viewed in a natural way in the context of the theory of $\H$-surfaces.  

\begin{theorem}\label{propositionpyo}
Let $v\in\s^2$ and $\H\in C^1(\s^2)$ such that  $\H(x)=0$ for every $x\in\s^2$  orthogonal to $v$.   Assume that $\H(-x)=-\H(x)$ for all $x\in\s^2$. Let  $\Gamma$ be a closed convex curve contained in $\Pi=v^\bot$ and let $\Omega\subset\Pi$ denote the domain bounded by $\Gamma$. If $M$ is an $\H$-surface spanning $\Gamma$, then $\Gamma$ is contained in the solid cylinder $C_\Omega=\Omega\times\R v$. If, in addition, $\Gamma$ is a circle and $\H$ is rotational about $v$, then $M$ is a rotational surface.  
\end{theorem}

Notice that we do not require in the hypothesis that $M$ is embedded. 

\begin{proof}
Up to a change of coordinates, we assume $v=e_3$ and that $\Pi$ is the plane of equation $x_3=0$. Hence, $\H$ vanishes on the equator $\langle x,e_3\rangle=0$ and in particular any vertical plane is an $\H$-surface. From the hypothesis $\H(-x)=-\H(x),\ \forall x\in\s^2$ we know that the tangency principle \ref{ppiotangencia} applies. We define the solid cylinder   $C_\Omega=\Omega\times\R$.

First, we show that $\mathrm{int}(M)$ lies inside $C_\Omega$. By contradiction, assume that there exists some $p\in\mathrm{int}(M)$ lying outside $C_\Omega$. Let $\hat{p}$ be the orthogonal projection of $p$ onto  $\Pi$, so $\hat{p}\not\in\Omega$. Let $q\in \Omega$ be a point of $\Omega$ realizing the distance to $\hat{p}$. If $P$ is the vertical plane with $q\in P$, then $\Omega$ lies contained in one side of $P$ because $\Gamma$ is convex. Moving $P$ outwards $\Omega$ and parallel in the direction of its normal vector, let $P_0$ be the last position that  touches $M$. Then $P_0$ lies at one side of $M$ and the  tangency principle yields a contradiction.

Once we have proved that  $\mathrm{int}(M)$ lies inside $C_\Omega$, next we show that $M$ is a graph over $\Omega$, which in particular yields that $M$ is embedded. Let us denote the vertical translations  $M_s=M+se_3,\ s\geq0$, which are again $\H$-surfaces. By compactness and for $s$ large enough, $M_s$ and $M$ are disjoint. We decrease $s$ until reaching a  time $s_0$ where $M_{s_0}$ meets $M$ at the first time. It is not possible that $s_0>0$ because otherwise, the tangency principle implies that $M_{s_0}=M$ and we arrive to the contradiction $\partial M_{s_0}=\Gamma+s_0 e_3\not=\Gamma$. Thus necessarily by vertical translations, the first contact point occurs at $s=0$, that is, $M$ comes to its initial position. Arguing similarly with vertical translation of type $M-s e_3$, $s\geq0$, we conclude that $M$ is a graph.   

In case that $\Gamma$ is a circle and $\H$ is rotational, we use Cor. \ref{cor-2}. 
\end{proof}

As we stated, Th.  \ref{propositionpyo} generalizes the main result of \cite{Pyo} for translating solitons, since $\H(x)=\langle x,v\rangle$, and there is a large family of functions $\H$ in the conditions of the theorem. For instance, given a $C^1$ odd function $f(t),\ t\in[-1,1]$, the function $\H(x)=f(\langle x,v\rangle)$, for a fixed $v\in\s^2$, lies in the hypothesis of Th. \ref{propositionpyo}.

%%%%%%%%%%%%%%%%%%%%%%%
\section{A height   estimate in terms of the height}\label{sec5}
%%%%%%%%%%%%%%%%%%%%%%%%%%%%%%%%%%

We obtain a relation  between the area of a compact $\H$-surface with planar boundary and the height  to the plane where its boundary lies. This result extends to the class of $\H$-surfaces others concerning the height and area of CMC surfaces in $\r3$ \cite{LoMo}, product spaces $M^2\times\R$ \cite{LeRo}, and of $\lambda$-translators in $\r3$ \cite{Lop}. In our result, we need to assume that the surface is a graph.

\begin{theorem}\label{thalturaarea}
Let   $\H>0$ and $M$ be a compact $\H$-graph on a plane $\Pi$. Suppose that the boundary of $M$ is contained in $\Pi$. If $h$ is the height of $M$ with respect to $\Pi$, then  
$$
h\leq\frac{\H_{\max}\, \mathrm{area}(M)}{2\pi},
$$
where   $\H_{\max}=\max \{\H(x):x\in\mathbb{S}^2\}$.  
\end{theorem}

\begin{proof}
Without loss of generality, we assume that $\Pi$ is the horizontal plane of equation $x_3=0$.   By Cor. \ref{corolariografoladoplano} , we know that $M$ lies at one side of $\Pi$, say $\{x\in\R^3:x_3>0\}$. Since $\H>0$, the Gauss map $N$ of $M$ points downwards, that is,  $N_3=\langle N,e_3\rangle<0$. In this setting, the height   of $M$ is $h=\max\{x_3(p): p\in M\}$.

For each $t>0$, we define $M(t)=M\cap\{x\in\R^3:x_3\geq t\}$ and $\Pi(t)=\{x\in\R^3:x_3=t\}$. Let $A(t)$ be the area of $M(t)$ and $\Gamma(t)=M(t)\cap\Pi(t)$. From the coarea formula   \cite[Th. 5.8]{Sak},
$$
A'(t)=-\int_{\Gamma(t)}\frac{1}{|\nabla x_3|}\,ds_t,
$$
where $ds_t$ is the line element of $\Gamma(t)$. Let us denote by $L(t)$ to the length of $\Gamma(t)$. On the one hand, by the Cauchy-Schwarz inequality we get
$$
L(t)^2=\left(\int_{\Gamma(t)}\,ds_t\right)^2\leq\int_{\Gamma(t)}\frac{1}{|\nabla x_3|}\,ds_t\int_{\Gamma(t)}|\nabla x_3|\,ds_t=-A'(t)\int_{\Gamma(t)}|\nabla x_3|\,ds_t,
$$
where $ds_t$ stands for the line element of $\Gamma(t)$.

Since $\Gamma(t)$ is a level curve of $x_3$ and $M(t)\subset\{x\in\R^3:x_3\geq t\}$, we have $|\nabla x_3|_{|\Gamma(t)}=\langle\nu(t),e_3\rangle$, where $\nu(t)$ is the unit inner conormal of $M(t)$ along $\Gamma(t)$. Therefore,
\begin{equation}\label{eq1}
L(t)^2\leq-A'(t)\int_{\Gamma(t)}\langle\nu(t),e_3\rangle\,ds_t.
\end{equation}
Because $\Delta x_3=2HN_3$, the divergence theorem yields
$$
\int_{M(t)}2 H N_3\,dM_t=-\int_{\Gamma(t)}\langle\nu(t),e_3\rangle\,ds_t.
$$
We substitute this integral in \eqref{eq1}. Using that $A'(t)\leq 0$ and $N_3\leq 0$, we have  
$$
L(t)^2\leq A'(t)\int_{M(t)}2 H N_3\,dM_t\leq 2\H_{\max}A'(t)\int_{M(t)}N_3\,dM_t.
$$
The curve $\Gamma(t)$ is possibly non-connected, hence the inner region bounded by it, $\Omega(t)$ consists on a finite union of compact, connected domains, $\Omega_i(t),\ i=1,...,n_t$, in $\Pi(t)$. By applying the divergence theorem to the vector field $e_3$ in the, possibly disconnected, 3-domain $W(t)$ bounded by $M(t)$ and $\Omega(t)$, we arrive to
\begin{equation}\label{eq2}
L(t)^2\leq-2\H_{\max}A'(t)\mathrm{area}(\Omega(t)).
\end{equation}
By denoting $L_i(t)$ to the length of $\partial\Omega_i(t)$, the classical isoperimetric inequality gives
$$
L(t)^2\geq\sum_{i=1}^{n_t}L_i(t)^2\geq 4\pi\sum_{i=1}^{n_t}\mathrm{area}(\Omega_i(t))=4\pi\,\mathrm{area}(\Omega(t)).
$$
Plugging this inequality in \eqref{eq2}, we deduce
$$
2\pi\leq-\H_{\max}A'(t).
$$
By integrating this inequality from $t=0$ to $t=h$ and taking into account that $A(0)$ is the area of $M$, we conclude
$$
h\leq\frac{\H_{\max}\, \mbox{area}(M)}{2\pi},
$$
which proves the result.
\end{proof}

Comparing the statement of Th. \ref{thalturaarea} with the CMC case \cite{LoMo}, the absence of a flux formula for $\H$-surfaces makes that in \eqref{eq1} and for an $\H$-surface not necessarily a graph, we cannot express the integral $\int_{\Gamma(t)}\langle\nu(t),e_3\rangle\,ds_t$ in terms of $\H_{max}$ and $\mbox{area}(\Omega(t))$. Consequently, we need to consider $M$ a graph in the statement of Th. \ref{thalturaarea}.

\section*{Acknowledgements}  

Antonio Bueno has been partially supported by the Projects P18-FR-4049 and Fundación Séneca 21937/PI/22. Rafael L\'opez is a member of the Institute of Mathematics  of the University of Granada and he   has been partially supported by  the Projects  PID2020-117868GB-I00 and MCIN/AEI/10.13039/501100011033.

\end{document}